\documentclass[12pt,twoside]{amsart}
\usepackage{amssymb,amsmath,amsthm, amscd, enumerate, mathrsfs, calligra}
\usepackage[all]{xy}
\usepackage{comment}
\usepackage{color}

\title[Q-abelian and $\bQ$-Fano finite quotients of abelian varieties]
{Q-abelian and $\bQ$-Fano finite quotients of abelian varieties}
\author[T.~Shibata]
{Takahiro Shibata}
\date{}

\address{National University of Singapore, Singapore 119076, Republic of Singapore}
\email{mattash@nus.edu.sg}

\DeclareMathOperator{\age}{age}
\DeclareMathOperator{\Alb}{Alb}

\DeclareMathOperator{\Aut}{Aut}

\DeclareMathOperator{\Fix}{Fix}
\DeclareMathOperator{\GL}{GL}

\DeclareMathOperator{\grp}{grp}

\DeclareMathOperator{\id}{id}

\DeclareMathOperator{\Ker}{Ker}

\DeclareMathOperator{\pr}{pr}

\DeclareMathOperator{\Stab}{Stab}

\DeclareMathOperator{\tor}{tor}

\DeclareMathOperator{\sHom}{\mathscr{H}\kern -.3pt \mathit{om}}

\newcommand{\bC}{\mathbb{C}}

\newcommand{\bP}{\mathbb{P}}
\newcommand{\bQ}{\mathbb{Q}}

\newcommand{\bZ}{\mathbb{Z}}

  {\end{list}}

\newtheorem{thm}{Theorem}[section]
\newtheorem{lem}[thm]{Lemma}
\newtheorem{cor}[thm]{Corollary}
\newtheorem{prop}[thm]{Proposition}

\newtheorem{ques}[thm]{Question}

\theoremstyle{definition}
\newtheorem{defn}[thm]{Definition}
\newtheorem{rmk}[thm]{Remark}
\newtheorem{example}[thm]{Example}
\newtheorem*{ack}{Acknowledgments}

\newtheorem*{notation}{Notation and Conventions}

\begin{document}

\begin{abstract}
We study finite quotients of abelian varieties (fqav for short) i.e.~quotients of abelian varieties by finite groups.
We show that Q-abelian varieties (i.e.~fqav's with $\bQ$-linearly trivial canonical divisors) are characterized by the existence of quasi\'etale polarized (or int-amplified) endomorphisms.
We show that every fqav has a finite quasi\'etale cover by the product of an abelian variety and a $\bQ$-Fano fqav.
Using such coverings,
we give a characterization of $\bQ$-Fano fqav's,
and show that $\bQ$-Fano fqav's and Q-abelian varieties are ``building blocks'' of general fqav's.
\end{abstract}

\maketitle

\setcounter{tocdepth}{1}

%
%
%
%
\section{Introduction}\label{sec:intro}

The aim of this article is to study \textit{finite quotients of abelian varieties} (\textit{fqav} for short),
that is,
quotients of abelian varieties by finite groups.
Such varieties have been studied since a long time ago.
For example,
Kummer surfaces are classical objects in algebraic geometry,
which are the quotients of abelian surfaces by involution.
Hyperelliptic surfaces,
the \'etale quotients of abelian surfaces which are not abelian, 
are also classical and appear in the Enriques--Kodaira classification of surfaces.

Let us see some of recent works on fqav's.
Yoshihara \cite{Yos95} gave a classification of finite quotients of abelian surfaces.
Auffarth--Lucchini Arteche \cite{ALA20} and Auffarth--Lucchini Arteche--Quezada \cite{ALAQ18} proved that,
if a finite group action on an abelian variety fixes the origin and induces an irreducible action on the tangent space at the origin,
then the quotient is a projective space.
Koll\'ar--Larsen \cite{KL09} proved that any finite quotient of a simple abelian variety of dimension $\ge 4$ is non-uniruled.
Lange \cite{Lan01} and Catenese--Demleitner \cite{CD20} completed a classification of hyperelliptic threefolds i.e.~\'etale quotients of abelian threefolds which are not abelian.

When studying a normal projective variety $X$,
it is fundamental to investigate positivity of the canonical divisor $K_X$.
There are two extreme classes of fqav's based on positivity of $K_X$;
one is the class of \textit{Q-abelian varieties}, which are fqav's with $\bQ$-linearly trivial canonical divisors,
and the other is the class of \textit{$\bQ$-Fano fqav's}, which are fqav's with ample anti-canonical divisors.
We mainly study those two classes in this article.

In Section \ref{sec:q-ab}, we study Q-abelian varieties.
The notion of Q-abelian variety is defined by Nakayama--Zhang \cite[Definition 2.13]{NZ10}.
We show that Q-abelian varieties are characterized by the existence of quasi\'etale polarized (or int-amplified) endomorphisms:

\begin{thm}[Theorem \ref{thm:q-ab}]\label{thm:main1}
Let $X$ be a normal projective variety with klt singularities.
Then the following are equivalent:
\begin{itemize}
\item[(1)] $X$ is a Q-abelian variety.
\item[(2)] $X$ admits a quasi\'etale polarized endomorphism.
\item[(3)] $X$ admits a quasi\'etale int-amplified endomorphism.
\end{itemize}
\end{thm}

Note that the implication $(2) \Rightarrow (3)$ is clear since a polarized endomorphism is int-amplified.
The implication $(3) \Rightarrow (1)$ follows  immediately from Meng \cite[Theorem 5.2]{Men20}.
So it is enough to show that any Q-abelian variety admits a polarized endomorphism.
In fact,
we will show that any fqav admits a polarized endomorphism, following Fakhruddin \cite{Fak03} (cf.~Proposition \ref{prop:pol} below).

In Section \ref{sec:fano}, we study $\bQ$-Fano fqav's,
and give a structure theorem of general fqav's:

\begin{thm}[Theorem {\ref{thm:decomp}}]\label{thm:main2}
Let $X$ be an fqav.
Then there exists a finite quasi\'etale cover $\theta: B \times F \to X$ such that  
$B$ is an abelian variety and $F$ is a $\bQ$-Fano fqav.
\end{thm}

Let us compare this theorem with other known results.
Yoshikawa \cite{Yos21} proved that any klt normal projective variety with an int-amplified endomorphism admits a finite quasi\'etale cover by a variety whose Albanese morphism is surjective and of Fano type (cf.~\cite[Theorem 1.3]{Yos21}).
We might regard Theorem \ref{thm:main2} as a precise form of \cite[Theorem 1.3]{Yos21} restricted to fqav's,
since any fqav admits a polarized (and so int-amplified) endomorphism by Proposition \ref{prop:pol} below.

Moreover, Theorem \ref{thm:main2} might also be seen as a precise form of a theorem of Matsumura--Wang \cite[Theorem 1.1]{MW21} restricted to fqav's.
Their theorem says, roughly speaking, that a klt pair with nef anti-log canonical divisor admits a finite quasi\'etale cover which has a \textit{locally constant} MRC fibration.
Note that any fqav has nef anti-canonical divisor (cf.~Lemma \ref{lem:nef} below).
In our Theorem \ref{thm:main2},
$Y=B \times F$ has a trivial MRC fibration $\pr_1: Y \to B$ (the canonical projection)  which is in particular locally constant (cf.~\cite[Definition 2.3]{MW21}).

As a corollary of Theorem \ref{thm:main2}, we give a characterization of $\bQ$-Fano fqav's:

\begin{cor}[Corollary \ref{cor:fano}]\label{cor:main2-1}
Let $X$ be an fqav.
Then the following are equivalent:
\begin{itemize}
\item[(1)] $q^\circ(X)=0$.
\item[(2)] $X$ is $\bQ$-Fano.
\item[(3)] $X$ is of Fano type.
\end{itemize}
\end{cor}

Here $q^\circ(X)$ denotes the maximum of the irregularities of all finite quasi\'etale covers of $X$ (cf.~Definition \ref{defn:qcirc} below).

We give another corollary, which says, roughly speaking, 
that Q-abelian varieties and $\bQ$-Fano fqav's are ``building blocks'' of general fqav's:

\begin{cor}[Corollary \ref{cor:fib}]\label{cor:main2-2}
Let $X$ be an fqav.
Then we have a sequence of morphisms 
$$X=X_0 \overset{q_0}{\to} X_1 \overset{q_1}{\to} \cdots \overset{q_{r-1}}{\to} X_r=Y$$
such that 
\begin{itemize}
\item $X_i$ is an fqav for any $i$.
\item For $0 \le i \le r-1$,
$q_i: X_i \to X_{i+1}$ is a surjective morphism with connected fibers and it fits into the following diagram.
$$
\xymatrix@C=40pt{
 B_i \times F_i \ar[r]^{\pr_1} \ar[d]_{\theta_i} & B_i \ar[d]^{\psi_i}  \\
 X_{i} \ar[r]^{q_i}  & X_{i+1}
}
$$
Here $B_i$ is an abelian variety, $F_i$ is a $\bQ$-Fano fqav,
$\pr_1$ is the projection,
$\theta_i$ is a finite quasi\'etale cover,
and $\psi_i$ is a finite cover.
\item $Y$ is a Q-abelian variety (possibly $Y$ is a point).
\end{itemize}
\end{cor}

Thus any fqav admits a tower fibrations whose fibers are constructed from $\bQ$-Fano fqav's and the base of the last fibration is a Q-abelian variety.

Finally, in Section \ref{sec:ex},
we construct some examples of fqav's,
focusing on relationships between positivity of canonical divisors and rationality of the varieties.
Recall the following problem.

\begin{ques}[cf.~{\cite[Question 4.4]{Fak03}}]\label{ques:fak}
Is a smooth projective rationally connected variety admitting a polarized endomorphism a toric variety?
\end{ques}

We construct rational fqav's which are rational but not of Fano type.
It is well-known that a projective toric variety is of Fano type,
so we see that the smoothness assumption in Question \ref{ques:fak} is essential.
Those examples also show that an fqav might be rational, regardless of the positivity of its anti-canonical divisor.

\begin{ack}
I thank Professor De-Qi Zhang for valuable discussions and comments and Yohsuke Matsuzawa for answering my questions.
I am supported by a Research Fellowship of NUS.
\end{ack}

\section{Preliminaries}\label{sec:prelim}

\begin{notation}
\begin{itemize}
$ \, $
\item Throughout this article, we work over the complex number field $\bC$.

\item For definitions of Kawamata log terminal (klt) singularities and log canonical (lc) singularities etc., 
we refer to \cite{KM98}.
Note that we always assume that the boundary divisor $\Delta$ of a log pair $(X, \Delta)$ is an effective $\bQ$-divisor.


\item A \textit{$\bQ$-Fano variety} is a normal projective variety $X$ which has klt singularities and an ample anti-canonical divisor.

\item A \textit{log Fano pair} is a klt pair $(X,\Delta)$ such that $-(K_X+\Delta)$ is ample.

\item A \textit{variety of Fano type} is a normal projective variety $X$ such that there is a boundary $\bQ$-divisor $\Delta$ on $X$ such that $(X,\Delta)$ is log Fano.

%

\item For a normal projective variety $X$,
$q(X)=h^1(X, \mathcal O_X)$ denotes the \textit{irregularity} of $X$.

\item For a morphism $f: X \to X$ from a variety $X$ to itself,
$\Fix(f)$ denotes the set of $f$-fixed points on $X$.

\item For a variety $X$, $\Aut(X)$ denotes the automorphism group of $X$.
When $X$ is an algebraic group, $\Aut_{\grp}(X)$ denotes the group of automorphisms of algebraic group.


\item For an abelian variety $A$,
$A_{\tor}$ denotes the set of torsion points on $A$.

\item For an abelian variety $A$ and a point $a \in A$,
$t_a: A \to A$ denotes the translation map by $a$.

\item Let $A$ be an abelian variety and a morphism $g: A \to A$.
Then we can write $g=t_a \circ \phi$ with $\phi: A \to A$ being a group homomorphism.
We say that $\phi$ is the \textit{holonomy part} of $g$.

Given a group $G < \Aut(A)$, the set $G_0 \subset \Aut_{\grp}(A)$ of the holonomy parts of the elements of $G$ forms a group.
We say that $G_0$ is the \textit{holonomy part} of $G$.


\end{itemize}

\end{notation}

In what follows,
we recall some basic notions and prove some lemmas which are needed in later sections.

First we recall the notions of polarized endomorphism and int-amplified endomorphism.

\begin{defn}\label{defn:endo}
Let $X$ be a projective variety.

(1) 
A \textit{polarized endomorphism} on $X$ is a finite surjective morphism $f: X \to X$ such that $f^*H \sim dH$ for some ample divisor $H$ and some $d >1$.
Note that we do NOT regard $f$ as a polarized endomorphism when $f^*H \sim H$.

(2) 
An \textit{int-amplified endomorphism} on $X$ is a finite surjective morphism $f: X \to X$ such that $f^*H -H$ is ample for some ample divisor $H$.
\end{defn}

Note that a polarized endomorphism is int-amplified.
For details on int-amplified endomorphisms, see \cite{Men20}.

Next, we define some basic notions on finite covers of normal varieties.

\begin{defn}\label{defn:cover}
Let $f: X \to Y$ be a finite surjective morphism of normal varieties.
We call $f$ a \textit{finite cover} of $X$.
\begin{itemize}
\item We denote the \textit{ramification divisor} of $f$ by $R_f$.

\item We denote the support of $R_f$ by $\tilde{R}_f$.

\item Set $\tilde{B}_f=f(\tilde{R}_f) \subset Y$, which is a closed subset of pure codimension 1.

\item We say that $f$ is \textit{quasi\'etale} if $R_f=0$.

\item We say that $f$ is \textit{Galois} if the field extension $k(X)/k(Y)$ is a finite Galois extension.
This is equivalent to say that $f$ is the quotient map by a finite group acting on $X$.
\end{itemize}
\end{defn}

Let us recall the definition of Q-abelian varieties.

\begin{defn}\label{defn:q-ab}
A \textit{Q-abelian variety} is a normal projective variety $X$ which admits a finite quasi\'etale cover $\pi: A \to X$ with $A$ being an abelian variety.
\end{defn}

Taking a Galois closure, a Q-abelian variety always admits a finite quasi\'etale Galois cover by  an abelian variety.
So any Q-abelian variety is an fqav.
It is clear that an fqav is Q-abelian if and only if it has $\bQ$-linearly trivial canonical divisor.

There is a characterization of uniruledness of Q-abelian varieties due to Koll\'ar--Larsen.

\begin{defn}[cf.~{\cite[Definition 1]{KL09}}]
(1) Let $V$ be an $n$-dimensional complex vector space.
For any $f \in \GL(V)$ of finite order, the \textit{age} of $f$ is the following quantity:
$$\age(f)=r_1+\cdots+r_n,$$
where $e^{2\pi \sqrt{-1} r_1},\ldots,e^{2 \pi \sqrt{-1} r_n}$ are the eigenvalues of $f$ with multiplicity and $0 \le r_i <1$ for all $i$.

(2) Let $\rho: G \to \GL(V)$ be a finite-dimensional linear representation of a finite group $G$.
We say that $\rho$ satisfies the (\textit{local}) \textit{Reid--Tai condition} if $\age(\rho(g)) \ge 1$ for any $g \in G$ such that $\rho(g) \neq \id_V$.

(2) Let $X$ be a smooth projective variety and $G < \Aut(X)$ a finite subgroup.
For any point $x \in X$, $\Stab_G(x)$ denotes the stabilizer of $x$ and $T_{X,x}$ denotes the tangent space of $X$ at $x$.
We say that the $G$-action on $X$ satisfies the (\textit{global}) \textit{Reid--Tai condition} if the induced action of $\Stab_G(x)$ on $T_{X,x}$ satisfies the local Reid--Tai condition for any $x \in X$.
\end{defn}

\begin{thm}[cf.~{\cite[Theorem 10]{KL09}}]\label{thm:kl}
Let $X=A/G$ be a Q-abelian variety.
Then the following are equivalent:
\begin{itemize}
\item[(1)] $X$ has canonical singularities.
\item[(2)] $\kappa(X)=0$.
\item[(3)] $X$ is non-uniruled.
\item[(4)] $G$ satisfies the Reid--Tai condition.
\end{itemize}
\end{thm}

\begin{rmk}
More generally, Koll\'ar--Larsen proved Theorem \ref{thm:kl} for finite quotients of projective varieties with canonical singularities and numerically trivial canonical divisors. For details, see \cite{KL09}.
\end{rmk}

The following lemma implies that any fqav naturally admits a structure of klt pair.

\begin{lem}\label{lem:klt}
Let $Y$ be a smooth variety of dimension $n$,
$G < \Aut(Y)$ a finite group,
and $\pi:Y \to X$ the quotient map by $G$.
Let $\tilde{B}_\pi=\bigcup_i B_i$ be the irreducible decomposition and let $e_i \in \bZ_{>1}$ the ramification index over $B_i$.
Set $\Delta=\sum_i (1-1/e_i)B_i$.
Then $(X,\Delta)$ is a klt pair satisfying $K_Y \sim \pi^*(K_X+\Delta)$.
\end{lem}

\begin{proof}
By the ramification formula, $K_Y \sim \pi^*K_X + \sum_{i,j} (e_i-1) R_{ij}$,
where $\{R_{ij} \}_j$ is the set of prime divisors over $B_i$.
Now $\pi^*B_i=\sum_j e_i R_{ij}$, 
so $K_Y \sim \pi^*(K_X+\Delta)$.
Hence $(X,\Delta)$ is klt by \cite[Proposition 5.20]{KM98}.
\end{proof}

The following lemma shows that any fqav has a nef anti-canonical divisor.

\begin{lem}\label{lem:nef}
Let $X$ be a normal projective variety and 
$\pi: A \to X$ be a finite cover by an abelian variety $A$.
Then $-K_X$ is nef.
\end{lem}

\begin{proof}
The ramification formula implies  $\pi^*(-K_X) \sim R_{\pi}$.
Now $R_{\pi}$ is an effective divisor on the abelian variety $A$, so it is nef.
Hence $-K_X$ is also nef.
\end{proof}

Recall the definition of $q^\circ(X)$ by Nakayama--Zhang \cite[Definition 2.6]{NZ10}.

\begin{defn}\label{defn:qcirc}
Let $X$ be a normal projective variety.
We define 
$$q^\circ(X)=\max_Y q(Y),$$
where $Y$ takes all normal projective varieties which are finite quasi\'etale covers of $X$.
\end{defn}

Note that $q^\circ(X)$ can be infinite.
For example, if $X$ is a curve of genus $\ge 2$,
then $X$ has an \'etale cover from a curve of arbitrarily large genus (e.g.~taking cyclic covers associated to a torsion line bundle of sufficiently large order).

Nevertheless, $q^\circ(X)$ is finite if $X$ is klt and $-K_X$ is nef.
Recall the following result by Q.~Zhang.

\begin{thm}[{\cite[Corollary 2]{QZ05}}]\label{thm:surj}
Let $(X, \Delta)$ be a projective lc pair such that $-(K_X+\Delta)$ is nef.
Then the Albanese morphism of $X$ is surjective with connected fibers.
\end{thm}

\begin{lem}
Let $X$ be a normal projective klt variety with $-K_X$ being nef.
Then $0 \le q^\circ(X) \le \dim X$.
\end{lem}

\begin{proof}
Let $Y$ be any normal projective variety which is a finite quasi\'etale cover of $X$.
Then $Y$ is also klt with $-K_Y$ being nef by the ramification formula.
So the Albanese morphism $a: Y \to \Alb(Y)$ is surjective by Theorem \ref{thm:surj}.
Now we have $q(Y)=\dim \Alb(Y)$ since $Y$ has rational singularities.
So $q(Y) \le \dim Y = \dim X$.
\end{proof}

Finally we prove some lemmas on abelian varieties which we use later.

\begin{lem}\label{lem:cut}
Let $A$ be an abelian variety and $B, C \subset A$ abelian subvarieties of $A$.
Assume that $\dim B=\dim A -1$ and $C \not\subset B$.
Let $D$ be the connected component of $B \cap C$ containing $0$.
Then $(B+p) \cap (C+q)$ is a finite union of translates of $D$ for any $p,q \in A$.
\end{lem}

\begin{proof}
We have 
$$(B+p) \cap (C+q)=((B+p-q) \cap C)+q,$$
so we may assume $q=0$.
Take the quotient map $\pi: A \to A/B$,
then 
$$(B+p)\cap C=(\pi|_C)^{-1}(\pi(p)).$$
Now $\pi|_C: C \to A/B$ is surjective by assumption.
So any fiber of $\pi|_C$ is a finite union of translates of the connected component of $(\pi|_C)^{-1}(0)=B \cap C$ containing 0, which is $D$.
\end{proof}

\begin{lem}\label{lem:kappa}
Let $A$ be an abelian variety of dimension $n$.
\begin{itemize}
\item[(1)] 
Let $T_1=A_1+a_1$ and $T_2=A_2+a_2$ be translates of abelian subvarieties of codimesion 1.
Then $T_1 \sim_{\bQ} T_2$ if and only if $A_1=A_2$ and $a_2-a_1 \in A_{\tor}$.
\item[(2)] 
Let $D=\sum_{i=1}^r m_i T_i$ be an effective $\bQ$-divisor, 
where each $T_i=A_i+a_i$ is a torsion translate of an abelian subvariety of codimension 1.
Then $$\kappa(D)=n- \dim \bigcap_{i=1}^r A_i.$$
Moreover, if $\kappa(D)=n$, then $D$ is ample 
and $A$ is isogenous to a product of elliptic curves.
\end{itemize}
\end{lem}

\begin{proof}
(1): If $A_1=A_2$, then $T_2=t_{a_2-a_1}^*T_1$, so $T_2-T_1 \sim_{\bQ} 0$.
Conversely, if $A_1 \neq A_2$, then consider the quotient map $p_1: A \to A/A_1$.
Then $T_1$ is $p$-trivial but $T_2$ is not.
So $T_1 \not\sim_{\bQ} T_2$.

(2): By (1), we may assume that $D=\sum_i A_i$.
Let $B$ be the connected component of $\bigcap_{i=1}^r A_i$ containing 0.
Set $k=\dim B$.
Let $\pi:A \to A/B$ be the quotient map and $\overline{A}_i=\pi(A_i)$.
Then $\overline{A}_i$ is a prime divisor on $A/B$ such that $\pi^*\overline{A}_i=A_i$.
If we have $\kappa(\sum_i \overline{A}_i)=\dim A/B =n-k$,
then 
$$\kappa(D)= \kappa(\sum_i A_i)=\kappa(\pi^*(\sum_i \overline{A}_i))=\kappa(\sum_i \overline{A}_i)=n-k.$$
So we may assume that $B=0$.

Then we can choose, say, $A_1, \ldots, A_n$ such that $\bigcap_{i=1}^n A_i$ is finite.
Note that $A_{i_1} \cap \cdots \cap A_{i_m}$ is a finite union of translates of an abelian subvariety for any $A_{i_1},\ldots, A_{i_m}$ due to Lemma \ref{lem:cut}.
Let $p_i:A \to E=A/A_i$ be the quotient map for $1 \le i \le n$ and 
$p=(p_1,\ldots,p_n):A \to E_1 \times \cdots \times E_n$ the induced morphism.
Now $\Ker(p)=\bigcap_{i=1}^n A_i$ has dimension 0 and $\dim A= \dim ( E_1 \times \cdots \times E_n)$,
so $p$ is an isogeny.
Therefore 
$$\mathcal O_A(\sum_i A_i)=p^*(\mathcal O_{E_1}(O_1) \boxtimes \cdots \boxtimes \mathcal O_{E_n}(O_n))$$
is ample,
where $O_i \in E_i$ is the origin. 
So $D$ is ample and $A$ is isogenous to a product of elliptic curves.
\end{proof}

\begin{lem}\label{lem:poincare}
Let $A$ be an abelian variety and $G < \Aut_{\grp}(A)$ a finite subgroup.
For any $G$-invariant abelian subvariety $B \subset A$,
there exists a $G$-invariant abelian subvariety $C \subset A$ such that the addition map $\mu: B \times C \to A$ is an isogeny.
\end{lem}

\begin{proof}
Fix a $G$-invariant ample line bundle $L$ on $A$ (e.g.~$L=\bigotimes_{g \in G}g^*M$ for some ample line bundle $M$).
Then the homomorphism $\phi_L: A \to \hat A$ given by $\phi_L(a)=t_a^*L \otimes L^{-1}$ is $G$-equivariant, 
giving a $G$-action on $\hat A$ by $g \cdot \alpha=(g^{-1})^*\alpha$.
Let $\iota: B \to A$ be the inclusion map and 
$C \subset A$ the connected component of $\Ker(\hat{\iota} \circ \phi_L)$ containing 0.
Now $\hat{\iota} \circ \phi_L$ is $G$-equivariant, so $C$ is $G$-invariant.
Moreover, $(\hat{\iota} \circ \phi_L)|_B=\phi_{L|_B}: B \to \hat{B}$ is an isogeny,
so $B,C \subset A$ satisfy the claim.
\end{proof}


%

%

%
%
%
%
\section{Q-abelian varieties}\label{sec:q-ab}

First, we prove the following.

\begin{prop}[cf.~{\cite[Theorem 4.2]{Fak03}}]\label{prop:pol}
Let $X$ be a finite quotient of an abelian variety $A$.
Then there exists a polarized endomorphism on $A$ which descends to a polarized endomorphism on $X$.
\end{prop}

\begin{rmk}
When $X$ is smooth, the assertion is contained in \cite[Theorem 4.2]{Fak03}.
But the proof works for any fqav.
\end{rmk}

We prove Proposition \ref{prop:pol}, following Fakhruddin \cite{Fak03}.

\begin{lem}\label{lem:trans}
Let $A$ be an abelian variety and $G < \Aut(A)$ a finite subgroup.
Then there exists a point $p \in A$ such that $p-g(p) \in A_{\tor}$ for any $g \in G$.
\end{lem}

\begin{proof}
Write $G=\{ t_{a_\lambda} \circ \varphi_{\lambda} \}_{\lambda \in \Lambda}$,
where $\varphi_{\lambda} \in \Aut_{\grp}(A)$.
Set $N=\# G$ and take $b_\lambda \in A$ such that $Nb_\lambda=a_\lambda$ for each $\lambda \in \Lambda$.
Setting $p=\sum_{\lambda \in \Lambda} b_\lambda$ and taking any $\mu \in \Lambda$, we have 
\begin{align*}
Ng_\mu(p) &= N(\varphi_\mu(\sum_{\lambda \in \Lambda} b_\lambda)+a_\mu) \\
&= \sum_{\lambda \in \Lambda} \varphi_\mu(a_\lambda)+Na_\mu \\
&= \sum_{\lambda \in \Lambda} g_\mu(a_\lambda) \\
&= \sum_{\lambda \in \Lambda} g_\mu(g_\lambda(0)) \\
&= \sum_{\lambda \in \Lambda} g_\lambda(0) \\
&= \sum_{\lambda \in \Lambda} a_\lambda = Np.
\end{align*}
So $p-g_\mu(p) \in A_{\tor}$ for any $\mu \in \Lambda$.
\end{proof}

\begin{cor}\label{cor:conj}
Let $A$ be an abelian variety and $G < \Aut(A)$ a finite subgroup.
Then there is a conjugate subgroup $G'< \Aut(A)$ of $G$ such that $g'(0) \in A_{\tor}$ for any $g' \in G'$.
\end{cor}

\begin{proof}
By Lemma \ref{lem:trans}, there is a point $p \in A$ such that $(t_p^{-1} \circ g \circ t_p)(0)=g(p)-p \in A_{\tor}$ for any $g \in G$.
So $G'=t_p^{-1}G t_p$ satisfies the claim.
\end{proof}

Using Proposition \ref{prop:pol},
we prove the following characterization of Q-abelian varieties.

\begin{proof}[Proof of Proposition \ref{prop:pol}]
By Corollary \ref{cor:conj}, we can write $X=A/G$ where $G < \Aut(A)$ is a finite subgroup with $g(0) \in A_{\tor}$ for all $g \in G$.
Then we can take $m \in \bZ_{>1}$ such that $mg(0)=g(0)$ for all $g \in G$.
Now the multiplication map $[m]:A \to A$ is commutative with $G$,
so $[m]$ descends to a morphism $f: X \to X$.
Since $[m]$ is polarized,
 $f$ is also polarized (cf.~\cite[Theorem 1.3]{MZ18}).
\end{proof}

Let us prove the main theorem of this section.

\begin{thm}\label{thm:q-ab}
Let $X$ be a normal projective variety with klt singularities.
Then the following are equivalent:
\begin{itemize}
\item[(1)] $X$ is a Q-abelian variety.
\item[(2)] $X$ admits a quasi\'etale polarized endomorphism.
\item[(3)] $X$ admits a quasi\'etale int-amplified endomorphism.
\end{itemize}
\end{thm}

\begin{proof}
Assume (1).
Take a finite quasi\'etale Galois cover $\pi: A \to X$ by an abelian variety $A$.
Then Proposition \ref{prop:pol} implies that a polarized endomorphism $g$ on $A$ descends to a polarized endomorphism $f$ on $X$.
Now $\pi$ and $g$ are quasi\'etale, so $f$ is also quasi\'etale.

Clearly (2) implies (3) since any polarized endomorphism is int-amplified.

Assume (3).
Let $f$ be a quasi\'etale int-amplified endomorphism on $X$.
Then we have $K_X \sim f^* K_X$.
If $K_X \not\equiv 0$, then the numerical class of $K_X$ in $N^1(X)$ is an eigenvector of $f^*$ with eigenvalue 1, which contradicts that $f$ is int-amplified (cf.~\cite[Theorem 1.1]{Men20}).
So $K_X \equiv 0$.
Then (1) follows from \cite[Theorem 5.2]{Men20}.
\end{proof}

On the other hand, there exists a smooth projective variety which admits a (non-quasi\'etale) int-amplified endomorphism but does not admit a polarized endomorphism.
The following example is told to the author by Y.~Matsuzawa.

\begin{example}
Let $E$ be an elliptic curve and $L$ a line bundle on $E$ such that  $\deg L=0$ and $L \not\sim_{\bQ} 0$.
Take the projective bundle $X=\bP(\mathcal O_E \oplus L)$ over $E$.
Then $X$ admits an int-amplified endomorphism (cf.~\cite[Section 7]{MY21}).
But $X$ does not admit a polarized endomorphism, according to a classification of surfaces admitting polarized endomorphisms due to S.-W. Zhang \cite[Proposition 2.3.1]{SWZ06}.
\end{example}

\section{$\bQ$-Fano finite quotients of abelian varieties and a structure theorem}\label{sec:fano}

Koll\'ar--Larsen proved that any non-quasi\'etale quotient of an abelian variety is uniruled:

\begin{lem}[{\cite[Theorem 10]{KL09}}]\label{lem:uniruled}
Let $X$ be a non-quasi\'etale finite quotient of an abelian variety.
Then $\kappa(X)=-\infty$ and $X$ is uniruled.
\end{lem}

The following proposition tells us when an fqav is $\bQ$-Fano.

\begin{lem}\label{lem:-k}
Let $A$ be an abelian variety of dimension $n$,
$G < \Aut(A)$ a finite group,
and $\pi:A \to X$ the quotient map by $G$.
\begin{itemize}
\item[(1)]
The reduced ramification divisor $\tilde{R}_{\pi}$ is of the form 
$$\tilde{R}_{\pi}=\bigcup_{i=1}^r (A_i+a_i ),$$ 
where $A_i \subset A$ is an abelian subvariety of codimension 1 and $a_i \in A$.
\item[(2)]
We have 
$$\kappa(-K_X)=n- \dim (\bigcap_{i=1}^r A_i).$$
Moreover, if $\kappa(-K_X)=n$, then $X$ is a $\bQ$-Fano variety and $A$ is isogenous to a product of elliptic curves.
\end{itemize}
\end{lem}

\begin{proof}
The ramification locus of $\pi$ is the union of the zero locus of $1-g$ for $g \in G \setminus \{ 1 \}$.
Fix $g \in G$ and write $g=t_a \circ \phi$ where $\phi$ is the holonomy part of $g$.
Then $$(1-g)^{-1}(0)=(1-\phi)^{-1}(a),$$
which is a finite union of translated abelian subvarieties if non-empty.
So the ramification locus of $\pi$ is a finite union of translated abelian subvarieties.
In particular, $\tilde{R}_\pi$ is a finite union of translated abelian subvarieties of codimension 1.

Replacing $G$ by its conjugate by applying Corollary \ref{cor:conj}, 
we may assume that 
$$\tilde{R}_{\pi}=\bigcup_{i=1}^r (A_i+a_i )$$ 
where $A_i \subset A$ is an abelian subvariety of codimension 1 and  $a_i \in A_{\tor}$ for all $i$.
By the ramification formula, $0 \sim \pi^*K_X+R_\pi$.
So $\kappa(-K_X)=\kappa(R_{\pi})$, and $-K_X$ is ample if and only if $R_{\pi}$ is.
Hence the assertion follows from Lemma \ref{lem:kappa}.
\end{proof}

%

Let us prove our main theorem:

\begin{thm}\label{thm:decomp}
Let $X$ be an fqav.
Then there exists a finite quasi\'etale cover $\theta: B \times F \to X$ such that 
$B$ is an abelian variety and $F$ is a $\bQ$-Fano fqav.
\end{thm}

\begin{proof}
Set $n=\dim A$ and let $\{ g_{\lambda} \}_{\lambda \in \Lambda} \subset G$ be the elements of $G$ satisfying $\dim \Fix(g_\lambda)=n-1$.
Let $\phi_\lambda$ be the holonomy part of $g_\lambda$.
Clearly $\Fix(hgh^{-1})=h(\Fix(g))$ for any $g,h \in G$,
so the subgroup $N < G$ generated by $\{ g_{\lambda} \}_{\lambda \in \Lambda}$ is normal.


Decompose $\pi$ as $A \overset{\pi_1}{\to} Y=A/N \overset{\pi_2}{\to} X$.
We show that $\pi_2$ is quasi\'etale.
Take any $g \in G \setminus N$ and $y \in Y$.
Take $a \in A$ such that $\pi_1(a)=y$.
Then we have 
\begin{align*}
g \cdot y=y &\Leftrightarrow \pi_1(g \cdot a)=\pi_1(a) \\
&\Leftrightarrow \exists h \in N,\ h \cdot g \cdot a=a \\
&\Leftrightarrow a \in \bigcup_{h \in N} \Fix(hg).
\end{align*}
So $$\Fix(g|_Y)=\pi_1(\bigcup_{h \in N} \Fix(hg)).$$
For any $h \in N$, $hg \not\in N$ and so $\dim \Fix(hg) \le n-2$.
Hence $\dim \Fix(g|_Y) \le n-2$ for any $g \in G \setminus N$.
This means that $\pi_2$ is quasi\'etale.

Let $N_0$ be the holonomy part of $N$ and let $B \subset A$ be the connected component of 
$$A^{N_0}=\bigcap_{\phi \in N_0}\Fix(\phi)=\bigcap_{\lambda \in \Lambda}\Fix(\phi_\lambda)$$ containing the origin.
If $B=0$, then $A/N$ is $\bQ$-Fano by Lemma \ref{lem:-k}.
If $B=A$, then $N$ is a group of translations and so $A/N$ is an abelian variety.
Hence we may assume that $0 < \dim B < \dim A$.
It follows from Lemma \ref{lem:poincare} that there exists an $N_0$-invariant abelian subvariety $C \subset A$ such that the addition map $\mu: B \times C \to A$ is an isogeny.
Set 
$$\tilde{N}=\{ t_b \times (t_c \circ \phi|_C) \in \Aut(B \times C) \mid \phi \in N_0,\ t_{b+c} \circ \phi \in N \}.$$
Then we obtain the following diagram:
$$
\xymatrix@C=40pt{
B \times C \ar[r]^{\tilde{q}} \ar[d]_{\mu} & (B \times C) / \tilde{N} \ar[d]^{\overline{\mu}} \\
A \ar[r]^q & A/N
}
$$
Here the horizontal morphisms are the quotient maps and $\overline{\mu}$ is the induced morphism.

Take $(y,z), (y',z') \in B \times C$ such that $q(y+z)=q(y'+z')$.
Then there exists $g \in N$ such that $g(y+z)=y'+z'$.
Write $g=t_a \circ \phi$ with $\phi \in N_0$ and $a=b_0+c_0$ with $b_0 \in B$, $c_0 \in C$.
Then 
$$y'+z'=y+\phi(z)+b_0+c_0.$$
Setting 
$$d=y+b_0-y' =z'-(\phi(z)+c_0) \in B \cap C,$$
we have $t_{b_0-d} \times (t_{c_0+d} \circ \phi|_C) \in \tilde{N}$ and 
$$(t_{b_0-d} \times (t_{c_0+d} \circ \phi|_C))(y,z)=(y+b_0-d, \phi(z)+c_0+d)=(y',z').$$
Thus $\overline{\mu}$ is bijective, so it is an isomorphism.

Set 
$$\tilde{N}_B=\{ t_b \in \Aut(B) \mid t_b \times h \in N \ \mathrm{for\ some\ }h \in \Aut(C) \}.$$
Then the projection $p:B \times C \to B$ descends to $\overline{p}:(B \times C)/\tilde{N} \to B/\tilde{N}_B$.
We have the following diagram:

$$
\xymatrix@C=40pt{
B \times C \ar@(d, ul)[ddr]_p \ar@(r, ul)[drr]^{\tilde{q}} \ar[dr]^{\nu} 
& & \\
 & Z \ar[r]^{q_Z} \ar[d]^{p_Z} & (B \times C) / \tilde{N} \ar[d]^{\overline{p}}\\
 & B \ar[r]^{q_B}  & B/\tilde{N}_B
}
$$
Here $q_B$ is the quotient map,
$p_Z$ is the base change of $\overline{p}$ along $q_B$ and $\nu$ is the induced morphism by the universal property.
Now $\nu$ is a factor of $\tilde{q}$, so $Z=(B \times C)/L$ for some 
subgroup $L < \tilde{N}$.
Any element of $L$ is of the form $\id_B \times h$ since $\nu$ factors through $p$.
So, setting $\tilde{N}_C=\{ g |_C \mid g \in N,\ g(C)=C \}$.
$L= \{ 1 \} \times \tilde{N}_C$ by the minimality of $Z$.
Hence $Z=B \times (C/\tilde{N}_C)$.
Note that $q_Z$ is quasi\'etale since it is the base change of the \'etale cover $q_B$.

As a result, we obtain a finite quasi\'etale cover 
$$Z=B \times (C/\tilde{N}_C)  \overset{q_Z}{\to} (B \times C)/\tilde{N} \cong A/N \overset{\pi_2}{\to} X.$$
If $C/\tilde{N}_C$ is not $\bQ$-Fano,
then do the same process to $C/\tilde{N}_C$.
Continuing this process, the assertion follows.
\end{proof}

As a corollary, we give a characterization of $\bQ$-Fano fqav's:

\begin{cor}\label{cor:fano}
Let $X$ be an fqav.
Then the following are equivalent:
\begin{itemize}
\item[(1)] $q^\circ(X)=0$.
\item[(2)] $X$ is $\bQ$-Fano.
\item[(3)] $X$ is of Fano type.
\end{itemize}
\end{cor}

\begin{proof}
Note that $X$ is klt due to Lemma \ref{lem:klt}.
Let $\theta: Y= B \times F \to X$ a finite quasi\'etale cover obtained in Theorem \ref{thm:decomp} and $\pr_1: Y \to B$, $\pr_2: Y \to F$ the canonical projections.
Then $K_X \sim \theta^*K_Y \sim \theta^* \pr_2^*K_F$ by the ramification formula.

If $q^\circ(X)=0$,
then $B=0$ since $\pr_1: Y \to B$ is the Albanese morphism of $Y$.
This implies that $-K_X \sim -\theta^*K_F$ is ample, so $X$ is $\bQ$-Fano.

Clearly (2) implies (3),
so we show that (3) implies (1).
Let $\phi: Y \to X$ be a finite quasi\'etale cover by a normal projective variety $Y$.
Take a boundary divisor $\Delta$ on $X$ such that $(X, \Delta)$ is log Fano.
Then $(Y, \pi^*\Delta)$ is also log Fano since $K_Y + \pi^*\Delta \sim \pi^*(K_X+\Delta)$.
So $q(Y)=0$ since $Y$ is rationally connected (cf.~\cite{QZ06}).
Therefore $q^\circ(X)=0$.
\end{proof}

For a general fqav, we have the following.

\begin{prop}\label{prop:q}
Let $X$ be an fqav.
Take a finite quasi\'etale cover $\theta: B \times F \to X$ such that $B$ is an abelian variety and $F$ is a $\bQ$-Fano fqav.
Then $q^\circ(X)=\dim B$.
\end{prop}

\begin{proof}
Take a finite quasi\'etale cover $\phi: Z \to X$ with $Z$ normal such that $q(Z)=q^\circ(X)$.
Taking $\phi$ sufficiently largely,
we may assume that $\phi$ factors as $\phi=\theta \circ \varepsilon$ for some $\varepsilon: Z \to B \times F$.
Then we have the following diagram.
$$
\xymatrix@C=40pt{
Z \ar[r]^\varepsilon \ar[d]_{p'} & B \times F \ar[r]^{\theta} \ar[d]_{\pr_1} & X  \\
B' \ar[r]^{\phi}  & B  & 
}
$$
Here $p': Z \to B'$ is the Albanese morphism of $Z$ and $\phi: B' \to B$ is the induced morphism.

Suppose $\dim B' > \dim B$.
Let $\{ b \} \times F$ be a general fiber of $\pr_1$.
Now $\dim \phi^{-1}(b) >0$ by assumption.
So an irreducible component $W$ of $\varepsilon^{-1}(\{b\} \times F)$ has positive irregularity.
But $\varepsilon|_W: W \to \{b\} \times F$ is a finite quasi\'etale cover, so $W$ is also $\bQ$-Fano.
This is a contradiction since the irregularity of any $\bQ$-Fano variety is zero.
\end{proof}

Proposition \ref{prop:q} implies the following characterization of Q-abelian varieties in terms of $q^\circ(X)$.

\begin{cor}\label{cor:charofq-ab}
Let $X$ be an fqav.
Then $X$ is a Q-abelian variety if and only if $q^\circ(X)=\dim X$.
\end{cor}

\begin{proof}
If $X$ is Q-abelian, then $X$ admits a finite quasi\'etale cover from an abelian variety and so $q^\circ(X)=\dim X$ due to Proposition \ref{prop:q}.

Conversely, if $q^\circ(X)=\dim X$,
then the finite quasi\'etale cover $\theta: B \times F \to X$ as in Theorem \ref{thm:decomp} satisfies 
$$\dim B = q^\circ(X)=\dim X$$
due to Proposition \ref{prop:q}.
So $F$ is just a point.
\end{proof}

As another corollary,
we see that any fqav is constructed from  Q-abelian varieties and $\bQ$-Fano fqav's:

\begin{cor}\label{cor:fib}
Let $X$ be an fqav.
Then we have the following diagram:
$$X=X_0 \overset{q_0}{\to} X_1 \overset{q_1}{\to} \cdots \overset{q_{r-1}}{\to} X_r=Y$$
such that 
\begin{itemize}
\item $X_i$ is an fqav for any $i$.
\item For $0 \le i \le r-1$,
$q_i: X_i \to X_{i+1}$ be a surjective morphism with connected fibers and it fits into the following diagram.
$$
\xymatrix@C=40pt{
 B_i \times F_i \ar[r]^{\pr_1} \ar[d]_{\theta_i} & B_i \ar[d]^{\psi_i}  \\
 X_{i} \ar[r]^{q_i}  & X_{i+1}
}
$$
Here $B_i$ is an abelian variety, $F_i$ is a $\bQ$-Fano fqav,
$\pr_1$ is the projection,
$\theta_i$ is a finite quasi\'etale cover,
and $\psi_i$ is a finite cover.
\item $Y$ is a Q-abelian variety (possibly $Y$ is a point).
\end{itemize}
\end{cor}

\begin{proof}
Let $\theta: B \times F \to X$ a finite quasi\'etale cover obtained in Theorem \ref{thm:decomp}.
Take a Galois closure 
$$\theta':Z \overset{\varepsilon}{\to} B \times F \overset{\theta}{\to} X$$
of $\theta$.
Let $p': Z \to B'$ be the Albanese morphism of $Z$.
Then $\dim B'=\dim B$ by Proposition \ref{prop:q}.
Let $\phi: B' \to B$ be the induced surjective morphism of Albanese varieties.
Then we obtain the following diagram.
$$
\xymatrix@C=40pt{
Z \ar[r]^\varepsilon \ar[d]_{p'} & B \times F \ar[d]^{\pr_1} \\
B' \ar[r]^{\phi}  & B
}
$$
Let $W$ be a general fiber of $p'$.
Then we have a finite quasi\'etale cover $\varepsilon|_W: W \to \{b\} \times F$ for some $b \in B$.
Now $F$ is $\bQ$-Fano, so $W$ is also $\bQ$-Fano.

Let $G$ be the Galois group of $\theta'$.
Then the action of $G$ on $Z$ descends to $B'$ since $B'$ is abelian and a general fiber of $p'$ is rationally connected.
So the fibration $p'$ is $G$-equivariant.
Taking the quotient $\psi': B' \to X_1=B'/G$, we obtain a surjective morphism $q_0: X \to X_1$ with connected fibers as the quotient of  the $G$-equivariant fibration $p'$.

Take any fiber $\{b\} \times F$ of $\pr_1$.
We compute 
\begin{align*}
(q_0 \circ \theta)(\{b\} \times F) 
&= (q_0 \circ \theta')(\varepsilon^{-1}(\{b\} \times F)) \\
&= (\psi' \circ p') ((p')^{-1}(\phi^{-1}(b))) \\
&= \psi'(\phi^{-1}(b)).
\end{align*}
So $(q_0 \circ \theta)(\{b\} \times F)$ has dimension zero, and hence it is a point.
The rigidity lemma implies that $q_0 \circ \theta$ factors as $q_0 \circ \theta=\psi \circ \pr_1$ for some $\psi: B \to X_1$.

$$
\xymatrix@C=40pt{
Z \ar@/^20pt/[rr]^{\theta'} \ar[r]^\varepsilon \ar[d]_{p'} & B \times F \ar[r]^{\theta} \ar[d]_{\pr_1} & X \ar[d]^{q_0}  \\
B' \ar[r]^{\phi} \ar@/_20pt/[rr]_{\psi'}  & B \ar[r]^{\psi}  & X_1
}
$$

Apply the same process to $X_1$ and so on, 
then we are done.
\end{proof}
\section{Examples}\label{sec:ex}

Classical Q-abelian varieties such as (singular) Kummer surfaces and hyperelliptic surfaces are non-uniruled.
But Theorem \ref{thm:kl} tells us that there may exist uniruled Q-abelian varieties.
In fact, there exists a \textit{rational} Q-abelian surface:

\begin{example}\label{example:ratq-ab}
Set $E_i=\bC/(\bZ \oplus \bZ \sqrt{-1})$,
which has an endomorphism $i:E_i \to E_i$ defined as the multiplication by $\sqrt{-1}$,
and $A=E_i \times E_i$.
Set $g=i \times i \in \Aut(A)$ and let $G < \Aut(A)$ be the subgroup generated by $g$.
Then $\Fix(g^k)$ is finite for any $k \in \bZ$,
so $X=A/G$ is Q-abelian.

The action of $\Stab_G(0)=G$ on $V=T_{A,0}$ is just the linear representation on $\bC^2$ by the group generated by the matrix
$$
M=\begin{pmatrix}
\sqrt{-1} & 0 \\
0 & \sqrt{-1} 
\end{pmatrix}.
$$
We have $\age(M)=1/2<1$,
so $X$ is uniruled (cf.~Theorem \ref{thm:kl}).

Moreover, it is easy to see that $V^G=0$.
So 
$$h^1(X,\mathcal O_X)=\dim H^1(A,\mathcal O_A)^G=\dim V^G=0.$$
Since $X$ has rational singularities,
any resolution $\tilde{X}$ of $X$ satisfies $q(\tilde{X})=q(X)=0$.
Hence $X$ is rational by Castelnuovo's rationality criterion.
\end{example}

There exists an fqav which is rational but not $\bQ$-Fano nor Q-abelian:

\begin{example}\label{example:nonfanorat}
Let $E$ be any ellptic curve and $E_i=\bC/(\bZ \oplus \bZ \sqrt{-1})$,
which has a complex multiplication $i:E_i \to E_i$ defined by the multiplication by $\sqrt{-1}$.
Set $A=E \times E_i$ and consider a finite group $G < \Aut_{\grp}(A)$ generated by $g=[-1] \times i$.
Let $\pi: A \to X$ be the quotient map by $G$.

Then $g^2=\id_E \times [-1] \in G$ is the only element of $G$ such that $\dim \Fix(g)=1$.
Now $\Fix(g)=E \times E_i[2]$ and $\kappa(-K_X)=1$.
So $X$ is not $\bQ$-Fano nor Q-abelian.

On the other hand,
the analytic representation of $g$ on $V=\bC^2$ is given by the matrix 
$$
\begin{pmatrix}
-1 & 0 \\
0 & \sqrt{-1} 
\end{pmatrix}
$$
So $V^G=0$, which implies that $q(X)=\dim V^G=0$.
So $X$ is rational by Castelnuovo's rationality criterion.
\end{example}

\begin{rmk}
Yoshikawa \cite{Yos21} proved that a smooth projective variety which is rationally connected and has an int-amplified endomorphism is of Fano type (cf.~\cite[Corollary 1.4]{Yos21}).
Example \ref{example:ratq-ab} and Example \ref{example:nonfanorat}  show that
we cannot drop the smoothness assumption in Yoshikawa's theorem.
\end{rmk}

%
%
%
%

In dimension two, a uniruled fqav with irregularity zero is rational due to Castelnuovo's criterion.
But it is uncertain (to the auther) whether the implication holds in higher dimension.
We conclude this section with the following question.

\begin{ques}
Is every $\bQ$-Fano fqav rational (or even toric)?
\end{ques}

\end{document}